\documentclass{ifacconf}

\usepackage{graphicx}      % include this line if your document contains figures
\usepackage{natbib}        % required for bibliography
\usepackage{subfigure}
\usepackage{color}  %给字体标红
\usepackage{amsfonts}  % \mathbb
\usepackage{amsmath}
\usepackage{amssymb}
\usepackage{tikz} %流程图（框图）
\usetikzlibrary{shapes.geometric, arrows,chains} %流程图（框图）
\usetikzlibrary{arrows.meta} %流程图（框图）
\usepackage{stfloats} % 单栏公式出现在某页的页底
\usepackage{makecell} % 表格
\usepackage{multirow}
\usepackage[inline]{enumitem} %item不换行
\usepackage{cases}
\newtheorem{Ass}{Assumption}
\newtheorem{Lem}{Lemma}
\newtheorem{Rem}{Remark}

\newtheorem{Def}{Definition}

\begin{document}
\begin{frontmatter}

\title{Optimal Parameter Design for DIGing on Minimizing Unweighted Sum of Squares\thanksref{footnoteinfo}} 

%\thanks[footnoteinfo]{Sponsor and financial support acknowledgment
%goes here. \textit{(Corresponding author: Li Chai.)}}
\thanks[footnoteinfo]{\textit{Corresponding author: Li Chai.}}
% Supported by Zhejiang Provincial Science Foundation (LZ24F030006), and Natural Science Foundation of China (U2441244, 62373323).} %Supported by National Natural Science Foundation of China under Grant U2441244 and 62373323.

\author[First]{Qiuchen Tian} 
\author[First]{Li Chai} %\corref{1}
%\authornote{aa}
\author[First]{Jinming Xu}

\address[First]{College of Control Science and Engineering,
	Zhejiang University, Hangzhou 310027, China\\
	(e-mails: $\left\{\right.$qiuchentian, chaili, jimmyxu$\left.\right\}$@zju.edu.cn)}
   
% step 3: use it where you need.
%\blfootnote{*Corresponding Author}

%\cortext[1]{Corresponding author}

%National Institute of Standards and Technology, 
%Boulder, CO 80305 USA (e-mail: author@ boulder.nist.gov).
%\address[Second]{Colorado State University, 
%   Fort Collins, CO 80523 USA (e-mail: author@lamar. colostate.edu)}
%\address[Third]{Electrical Engineering Department, 
%   Seoul National University, Seoul, Korea, (e-mail: author@snu.ac.kr)}

\begin{abstract}                % Abstract of 50--100 words
There is no general method for designing proper parameters to achieve faster convergence in distributed optimization algorithms. In this paper, we consider the distributed inexact gradient tracking (DIGing) algorithm with the objective function being the unweighted sum of squares. By representing the iteration algorithm as a dynamical linear system, we decompose it into different graph frequencies and obtain a set of decoupled subsystems, on which we can easily analyze the convergence rate. By using Routh stability criterion from control theory, we derive the explicit formula of the optimal worst-case convergence rate and the corresponding parameters. We can see that the convergence rate of DIGing is slow even for the simplest objective functions, thus acceleration is necessary for general application. 
%\textcolor{blue}{Finally, numerical examples verify the theoretical results.} 
The proposed method can be viewed as the first step toward optimal parameter design of DIGing algorithm in solving general objective functions.
%In this paper, we consider minimizing unweighted sum of squares by using the classic algorithm DIGing in decentralized optimization.
%%solving the problem of
%%using the classic decentralized optimization algorithm DIGing to solve average consensus problem, which is the simplest problem in optimization. 
%We give a new analysis approach by regarding the iteration in the algorithm as a discrete dynamic system and then provide a decomposition from the graph frequency domain.
%%provide a new approach for convergence analysis from the system perspective and use graph frequency domain decomposition. 
%In the analysis, we obtain the worst-case fastest convergence rate achievable and derive the corresponding optimal parameter design for DIGing. The experiment part verifies the optimality of the obtained parameters. % derived by us is optimal.
\end{abstract}

\begin{keyword}
	Algorithm DIGing, Worst-Case Convergence Rate, Optimal Parameter Design, Discrete Dynamic System, Routh's Stability Criterion
	
%Five to ten keywords, preferably chosen from the IFAC keyword list.
\end{keyword}

\end{frontmatter}
%===============================================================================

\section{Introduction}
Distributed optimization algorithms have drawn much attention in recent years, with extensive applications in several fields, such as distributed machine learning in \cite{nedic2017fast}, wireless networks in \cite{cohen2016distributed}, power system control in \cite{gan2012optimal}, etc. 
%Due to the need for fast computation, research on algorithm acceleration emerges.
Distributed inexact gradient tracking (DIGing) proposed by \cite{nedic2017achieving} is one of the most typical algorithms in the field of distributed optimization.
% and is widely cited.
%a classic algorithm in decentralized optimization. 
%This seminal work is widely cited in the field of distributed optimization research.
It uses gradient-tracking method in the iterations, so that the performance becomes comparable to centralized algorithms which enjoy linear convergence.
% eliminates the steady-state error in the case of using constant step size.
% in distributed optimization algorithm. 
%Thus, it enables the algorithm to converge to the optimum of the optimization problem.
%(It first eliminated the steady-state error in the case of using constant step size in distributed optimization algorithm, by using gradient tracking in iterations, which enables DIGing/the algorithm to converge to the optimum. )
%Before this,
Prior to this work, \cite{nedic2009distributed} provided the distributed subgradient descent (DGD) method. Though it is simple, it requires diminishing step sizes to attain the optimal solution, which consequently reduces the convergence rate.
% distributed optimization algorithm but it converges with a steady-state error, which could only be decayed by using decreased step size and thus slowing down the convergence rate.

Although the asymptotic convergence of DIGing has been proved, the convergence rate is usually slow in practical applications. Many efforts have been put on acceleration of distributed algorithms. The general method is to construct one type of Lyapunov functions to bound the error followed by measuring the bounds for the average descent.
%Since algorithm DIGing was proposed, subsequent efforts
%%numerous subsequent works have been devoted to/dedicated to leverage
%have focused on leveraging optimization techniques to accelerate convergence based on this foundation.
%Subsequent studies based on algorithm DIGing are quite extensive. 
%（此处补充1-2篇基于DIGing的前沿进展？）
\cite{qu2019accelerated} employed Nesterov's acceleration technique in distributed optimization algorithms.
\cite{scaman2017optimal,scaman2019optimal} developed dual-based methods to achieve optimal convergence rate.
%Some works developed dual-based methods to achieve optimal complexities, such as . Though accelerated, it is intractable in practice and sacrifices computation efficiency.
For dual-free methods, several works use inner-loops to accelerate the convergence, such as \cite{kovalev2020optimal,li2024accelerated}.
\cite{song2024optimal} provided the optimal single-loop gradient-tracking algorithm by using snapshots in the iterations.
% and Chebyshev communication acceleration. 
%Though these works 
%%above 
%developed several acceleration techniques, all of them use
%%are based on 
%the same analytical approach as in DIGing, which is constructing Lyapunov function to bound the consensus error and then measuring bounds for the average descent.
%% the descent of the average parts. 
%Though it is the main 
%%theoretical analysis 
%approach widely used in 
%%the field of 
%distributed optimization, the bounds obtained 
%%by it
%%this approach 
%are quite conservative.
%	using several techniques to bound the error of  
%	There are 丰富的后续工作基于DIGing; 选3-4个有代表性的工作。可以是learning/最优/最新等

Recently, techniques from robust control have been used to analyze the convergence rate of distributed optimization algorithms. Faster convergence rates and new insights have been reported.
%there are several
% have been several studies attempt to develop new analytical approach to attain much tight bounds on the convergence rate of optimization algorithms by using tools in control theory.
%%attempts focusing on analyzing optimization algorithms from control theory perspective. 
%\cite{lessard2022analysis} provides an analytical approach in dissipation for centralized optimization
%%non-distributed optimization/individual optimization/
%%(非分布式优化/单体优化)
%and derives the worst-case convergence rate.
%There are also 
%studies focusing on the analysis and parameter design for decentralized optimization. 
\cite{sundararajan2020analysis} provides a unified framework based on semi-definite programming (SDP) to analyze distributed optimization algorithms.
% and proposes a new algorithm SVL to converge faster.
\cite{van2022universal} presents an universal decomposition on the distributed optimization algorithms, but it fails to give the algorithm parameter design. \cite{zhang2024frequency} provides a frequency-domain framework for algorithm analysis from a robust control perspective, but its result obtained by using Nevanlinna–Pick interpolation requires high memory utilization and heavy communication burden.
%the parameter design given by them is not practical.
%/too ideal.
%	优化算法的分析方法，
However, the parameter design for distributed optimization algorithm still remains open. That is, little is known about how to find parameters for particular distributed algorithm to achieve faster convergence rate.

In this paper, we consider the problem of minimizing the simplest objective function, the unweighted sum of squares $\sum_{i=1}^N{\left( x-r_i \right) ^2}$ in a distributed manner.
%is a basic and important issue in decentralized optimization. 
We know that the optimal variable is $x^*=\left(1/N\right)\sum_{i=1}^N{r_i}$.
This problem is closely related to average consensus, which is of its own importance in multi-agent systems.
%The solution is closely related to average consensus.
%A same formulation of this problem is average consensus. 
The difference 
%between the two issues 
lies in that average consensus refers to seeking the average of initial values among agents through communications,
%by communications,
% without using the information of gradients,
while minimizing unweighted sum of squares
%in our work solving this problem 
permits the use of arbitrary initial values but the algorithm still converges to the fixed optimum.
% in the problem.
%, and the optimization algorithm, with the permission of using gradients in iteration, would converge to the fixed optimum in the problem.
%, with the permission of using gradients in iteration.
%It has been studied widely as the average consensus problem, which uses several agents to seek the average of their initial values through communications.
%Average consensus is the basic problem in distributed optimization and has been widely studied for a long time.
%Average consensus problem refers to using several agents to find the average of their initial values through communications.
Average consensus is a long-standing issue that has been studied for a long time.
\cite{xiao2004fast} gives the fast linear iterations for distributed averaging with using doubly stochastic gossip matrix.
%the analysis of the convergence  
\cite{yi2019average} presents an effective new approach to the analysis and design of consensus protocols in the graph spectrum domain. \cite{yi2023convergence} introduces a set of useful techniques to analyze the convergence rate of accelerated consensus and design the control protocols.
%first provide a new improved design of consensus protocols from the graph frequency domain by introducing a new parameter in communication, which expanded the range of available gossip matrix\textcolor{blue}{/loosened the assumptions on the gossip matrix} and accelerated the average consensus.
%Those research above are focusing on the problem of finding the average of the initial values for all agents. However, the average consensus problem can also be regarded as the simplest objective function in distributed optimization. In this perspective, algorithms designed for distributed optimization can also be used for solving average consensus problem. It is worth noting that in this way, the algorithm can use any random initial values and will still converge to the optimum for the average consensus problem.

%	Though
%Those subsequent works based on \cite{nedic2017achieving} mainly focus on using several techniques for accelerating the convergence rate of DIGing, and there are rare researches focusing on parameter design of DIGing. The parameter design of DIGing remains an open problem. 
%Consider parameter design for the classic algorithm DIGing 
In this paper, we use the unweighted sum of squares as the objective function to analyze the parameter design problem for DIGing.
%we provide a new approach for convergence analysis from the stability of discrete dynamic system perspective. 
%We derive the explicit formulation for the optimal parameters and the worst-case fastest rate achievable. % convergence rate. 
%%design for DIGing on solving the problem of minimizing unweighted sum of squares. 
%Though focusing on the simplest objective function, the analysis is not trivial.
By decomposing the iteration dynamics into different graph frequencies, we obtain a set of decoupled subsystems, from which we can easily analyze the convergence rate. By using Routh stability criterion from control theory, we derive the explicit formula of the  optimal worst-case convergence rate and the corresponding parameters. Even for the simplest objective functions, we can see that convergence rate of DIGing is quite slow and
further acceleration is necessary.
% the derivation is not trivial. 
The results can be viewed as the first step towards general methods for parameter design of distributed optimization algorithms.
%and sheds light on the analysis of DIGing for solving distributed optimization problems for general objective functions.

%Our contributions are summarized as follows. First, we propose a new approach for the convergence analysis of DIGing for solving average consensus problem. We start the analysis from the stability of discrete dynamic system perspective, and we use the graph frequency domain decomposition. Thus, we are able to transform the convergence analysis problem into a linear optimization problem subject to several quadratic inequality constraints. Second, by using the proposed analysis approach, we derive the optimal parameter design for DIGing. In the experimental part, we show that the obtained parameters are optimal.
%using the analysis approach proposed by us,

\section{Preliminary}
\subsection{Problem Formulation and Notations}
Consider the following optimization problem
%minimizing the unweighted sum of squares, which is formulated as
%commonly used average consensus problem in recent research(\cite{yi2019average,xiao2004fast}), which is formulated as
\begin{equation}
	\label{eq: problem average consensus}
	\underset{x\in \mathbb{R} ^d}{\min}\,\,\frac{1}{2}\sum_{i=1}^N{\left( x-r_i \right) ^2},
\end{equation}
where $r_i\in \mathbb{R} ^d$.
%, and $L>0$ is a constant. 
The optimum is $x^*=\frac{1}{N}\sum_{i=1}^N{r_i}$.
%	Notice that problem \eqref{eq: problem average consensus} is a special case of problem \eqref{eq: problem of quadratic}.
Let $x_i=\left[ x_{i}^{1},x_{i}^{2},\cdots ,x_{i}^{d} \right] ^T \in \mathbb{R}^d$ be a $local \,\, copy$ of the variable $x$ held by agent $i$. Denote its value at iteration $k$ as $x_i(k)$.
% We introduce the aggregate notation
Let $\boldsymbol{x}\triangleq \left[ x_{1}^{T},x_{2}^{T},\cdots ,x_{N}^{T} \right]^{T} \in \mathbb{R} ^{Nd}$ denote the augmented vector of all agents.
%, which is a column vector. 
%Denote the all-ones column vector as 
Throughout the paper, we denote $\mathbf{1}$ as the all one column vector with appropriate dimension, and
%Denote the identity matrix as
$I_N$ as the identity matrix of dimension $N$. The kronecker product is denoted as $\otimes$. 
%. Denote the kroneck product as 
%In this paper, for the sake of simplicity in notation, we denote $A^{\circ}\triangleq A\otimes I_d$ is a $M\times N$ matrix.

\subsection{Spectral Graph Theory} % Basic Knowledge of the Graph
%Consider all agents are embedded in a communication network, 
The communication network is modeled as an undirected graph $\mathcal{G} =\left( \mathcal{E} ,\mathcal{V},\mathcal{W}  \right) $ with nodes (agents) $\mathcal{V} =\left\{ v_1,v_2,\cdots ,v_N \right\} $ , edges $\mathcal{E} \subseteq \mathcal{V} \times \mathcal{V} $ and the weight matrix $\mathcal{W} \in \mathbb{R} ^{N\times N}$. %network/topology
%Self loops are allowed. 
Denote the communication matrix of the network induced by graph $\mathcal{G}$ as $W\in \mathbb{R}^{N\times N}$.
%, where $N$ is the number of nodes in the network. Denote the 
The Laplacian matrix is defined as $\mathcal{L} \triangleq D-W$, where $D=\mathrm{diag}\left\{ D_1,D_2,\cdots ,D_N \right\} $ is the degree matrix with
% the diagonal matrix with the diagonal elements being the degrees of each node.
%%whose main diagonal elements are the degrees of each node. %each/the individual nodes
%The degree of vertex $v_i$ is 
%%represented by 
$D_i=\sum_{j=1}^N{w_{ij}}$ for $i=1,\cdots,N$. 
%Denote $\bar{d}=\underset{i}{\max}\left\{ d_i \right\} $.
Throughout the article we make the following assumption.
%the following assumption is standard.

\begin{Ass}
	\label{ass graph}
	The weight matrix $W$ is a semi-definite positive matrix with nonnegative elements and the graph $\mathcal{G}$ is connected.
%	Consider the graph $\mathcal{G}$ is connected, and 
%	%		, i.e. there are no isolated points in graph $\mathcal{G}$. 
%	the matrix $W$ induced by graph $\mathcal{G}$ satisfies:
%	\begin{enumerate*}
%		\item $W=W^T$;
%		\item $W\succeq 0$.
%	\end{enumerate*}
\end{Ass}

%	With Assumption \ref{ass: xiaolin} holds, we immediately obtain the Laplacian matrix $\mathcal{L} =I-W$.

%	Note that the widely-used doubly-stochastic matrix is a special case of Assumption \ref{ass graph}.
\begin{Rem}
	\label{rem: xiaolin ass}
	In Assumption \ref{ass graph}, it only assumes the undirected graph $\mathcal{G}$ is connected with non-negative weights. It is a very weak assumption. %slack/weak/slight assumption.
	There is another assumption,
%	an another commonly used assumption, 
%	in decentralized optimization studies, 
	which is commonly used in much literature (\cite{xiao2004fast,nedic2017achieving}).
%	by \cite{xiao2004fast}. 
	%	In the following, we will give a commonly used assumption in decentralized optimization studies.
	They consider the matrix $W$ induced by graph $\mathcal{G}$ satisfying: 
	\begin{enumerate*}
		\item $W=W^T$;
		\item $W\mathbf{1}=W^T\mathbf{1}=\mathbf{1}$;
		\item $\rho \left( W-\frac{\mathbf{1}\mathbf{1}^T}{N} \right) <1$, where $\rho(\cdot)$ denotes the spectral radius of a matrix.
	\end{enumerate*}
	%	\begin{Ass}(\cite{xiao2004fast})
		%		\label{ass: xiaolin}
		%	\end{Ass}
%	Under this assumption, the Laplacian matrix is $\mathcal{L} =I-W$.
	%		Assumption \ref{ass graph} is weaker than the assumption on matrix $W$ induced by graph $\mathcal{G}$ in xiaolin2004 Th1, which is (i)$W\mathbf{1}=W^T\mathbf{1}=\mathbf{1}$;
	%		(ii)$\rho \left( W-\frac{\mathbf{1}\mathbf{1}^T}{N} \right) <1$, where $\rho(\cdot)$ denotes the spectral radius of a matrix.
	Note that this assumption
%	the assumption proposed by \cite{xiao2004fast} 
	implies that the entries of matrix $W$ satisfy $w_{ij}\in \left[ 0,1 \right] ,\, \forall i,j\in \left\{ 1,\cdots ,N \right\} $ and the graph $\mathcal{G}$ is connected.
%	It is obvious that it is more rigorous and is included in Assumption \ref{ass graph}.
	It is stronger than Assumption \ref{ass graph}.
%	a more rigorous assumption and is included in Assumption \ref{ass graph} as a special case.
	The matrix satisfying the assumption in \cite{xiao2004fast} is usually called doubly-stochastic. %\textit{}.
%	 "" and is a special case of Assumption \ref{ass graph}.
	%		However, Assumption \ref{ass graph} includes the situation of bipartite graphs while the assumption in xiaolin2004 does not. In the case of using the bipartite graph as the communication network, the algorithms AugDGM and DIGing diverge.
\end{Rem}

%\begin{Lem}[\cite{godsil2013algebraic}] %Algebraic Graph Theory,
%	\label{lemma 1}
%	For an undirected graph $\mathcal{G}$, the Laplacian matrix $\mathcal{L}$ satisfies the following properties.
%	\begin{enumerate}%[1)]
%		\item The Laplacian matrix $\mathcal{L}$ is symmetric and positive semi-definite.
%		\item All the eigenvalues of $\mathcal{L}$ are real and satisfy
%		$0=\lambda _1\leqslant \lambda _2\leqslant \cdots \leqslant \lambda _N\leqslant 2\bar{D}$, where $\bar{D}=\underset{i}{\max}\left\{ D_i \right\} $.
%		\item $\mathcal{L}$ has the following singular value decomposition:
%		$$\mathcal{L} =U\varLambda U^T,$$ where $\varLambda =\mathrm{diag} \{ \lambda _1,\lambda _2,\cdots ,\lambda _N \}  $ is a diagonal matrix, $U=\left[ u_1,u_2,\cdots ,u_N \right] \in \mathbb{R}^{N \times N} $ is an unitary matrix with column vectors $u_1,u_2,\cdots ,u_N\in \mathbb{R}^N$.
%		\item The graph $\mathcal{G}$ is $connected$ if and only if zero is a simple eigenvalue of $\mathcal{L}$, and the associated eigenvector is $\frac{1}{\sqrt{N}}\left[ 1,1,\cdots ,1 \right] ^T$.
%	\end{enumerate}
%	% connected是否需要解释？  corresponding/associated
%\end{Lem}

\begin{Lem}[\cite{godsil2013algebraic}] %Algebraic Graph Theory,
	\label{lemma 1}
	For an undirected graph $\mathcal{G}$, the Laplacian matrix $\mathcal{L}$ has the singular value decomposition $\mathcal{L} =U\varLambda U^T$, where $\varLambda =\mathrm{diag} \{ \lambda _1,\lambda _2,\cdots ,\lambda _N \}  $ is a diagonal matrix, $U=\left[ u_1,u_2,\cdots ,u_N \right] \in \mathbb{R}^{N \times N} $ is an unitary matrix with column vectors $u_1,u_2,\cdots ,u_N\in \mathbb{R}^N$.
	All the eigenvalues of $\mathcal{L}$ are real and satisfy
	$0=\lambda _1\leqslant \lambda _2\leqslant \cdots \leqslant \lambda _N\leqslant 2\bar{D}$, where $\bar{D}=\underset{i}{\max}\left\{ D_i \right\} $.
	The graph $\mathcal{G}$ is connected if and only if zero is a simple eigenvalue of $\mathcal{L}$, and the associated eigenvector is $\frac{1}{\sqrt{N}}\left[ 1,1,\cdots ,1 \right] ^T$.
%	\begin{enumerate}%[1)]
%		\item The Laplacian matrix $\mathcal{L}$ is symmetric and positive semi-definite.
%		\item 
%		\item $\mathcal{L}$ :
%		 
%		\item 
%	\end{enumerate}
%	% connected是否需要解释？  corresponding/associated
\end{Lem}

\section{DIGing algorithm for minimizing unweighted sum of squares}
%	\subsection{DIGing} Specialized algorithm
The original DIGing for solving unweighted sum of squares in problem \eqref{eq: problem average consensus} is written as follows
\begin{equation}
	\label{eq: original DIG entry}
	\left\{ \begin{split}
		x_i\left( k+1 \right) &=\sum_{j\in \mathcal{N} _i}{w_{ij}x_j\left( k \right)}-\alpha y_i\left( k \right)\\ 
		y_i\left( k+1 \right) &=\sum_{j\in \mathcal{N} _i}{w_{ij}y_j\left( k \right)}+x_i\left( k+1 \right) -x_i\left( k \right)
	\end{split},\right.
\end{equation}
%\begin{gather}
%	\label{eq: original DIG entry}
%	\begin{cases}
%		x_i\left( k+1 \right) =\sum_{j\in \mathcal{N} _i}{w_{ij}x_j\left( k \right)}-\alpha y_i\left( k \right)\\
%		y_i\left( k+1 \right) =\sum_{j\in \mathcal{N} _i}{w_{ij}y_j\left( k \right)}+L\left[ x_i\left( k+1 \right) -x_i\left( k \right) \right]\\
%	\end{cases},
%\end{gather}
where $W=\left[ w_{ij} \right] $ is the communication matrix. In \cite{nedic2017achieving}, 
%the weight matrix $W$ is assumed to be doubly-stochastic.
they assume matrix $W$ is doubly-stochastic. %/satisfies \textcolor{blue}{Assumption 2}
%It requires the weight matrix $W$ being doubly-stochastic/satisfying Assumption 2 in \cite{}.

In this paper, we consider a modified version of \eqref{eq: original DIG entry} by introducing an extra freedom $\varepsilon$ on the local gain $w_{ii}$ without changing the communication weights $w_{ij}, i\ne j$.
The modified DIGing algorithm is as follows
\begin{equation}
	\label{eq: entry_DIG with epsilon}
	\left\{ \begin{split}
		x_i\left( k+1 \right) &=(1-\varepsilon D_i) x_i\left( k \right)   +\varepsilon \sum_{j\in \mathcal{N} _i}{w_{ij}x_j\left( k \right)}  -\alpha y_i\left( k \right)\\ 
		y_i\left( k+1 \right) &=(1-\varepsilon D_i) y_i\left( k \right)  +\varepsilon \sum_{j\in \mathcal{N} _i}{w_{ij}y_j\left( k \right)} +\varDelta x_{i}^{k+1}  
	\end{split},\right.
\end{equation}
where $\varDelta x_{i}^{k+1}\triangleq x_i\left( k+1 \right) -x_i\left( k \right) $, and
$D_i=\sum_{j=1}^N{w_{ij}}$.
% $d_i$ is the degree of agent $i$.
For each agent $i$, the initialization of \eqref{eq: entry_DIG with epsilon} uses an arbitrary $x_i\left( 0 \right) \in \mathbb{R} ^d$ and sets $y_i\left( 0 \right) = x_i\left( 0 \right) -r_i $, for all $i=1,\cdots ,N$.

We will see that this setup leads to not only the weaker convergence condition in Assumption \ref{ass graph}, but also a faster convergence rate than the original algorithm \eqref{eq: original DIG entry}.

%The aggregated formulation of \eqref{eq: original DIG entry} is 
%\begin{equation*}
%	\begin{split}
%		\boldsymbol{x}\left( k+1 \right) &=W\otimes I_d \boldsymbol{x}\left( k \right) -\alpha \boldsymbol{y}\left( k \right),\\
%		\boldsymbol{y}\left( k+1 \right) &=W\otimes I_d \boldsymbol{y}\left( k \right) +\boldsymbol{x}\left( k+1 \right) -\boldsymbol{x}\left( k \right) .\\
%	\end{split}
%\end{equation*}
%\begin{gather}
%	\begin{cases}
%		\boldsymbol{x}\left( k+1 \right) =W^{\circ}\boldsymbol{x}\left( k \right) -\alpha \boldsymbol{y}\left( k \right)\\
%		\boldsymbol{y}\left( k+1 \right) =W^{\circ}\boldsymbol{y}\left( k \right) +LI_{Nd}\left[ \boldsymbol{x}\left( k+1 \right) -\boldsymbol{x}\left( k \right) \right]\\
%	\end{cases}.
%\end{gather}

Denote $\boldsymbol{x}\triangleq \left[ x_{1}^{T},x_{2}^{T},\cdots ,x_{N}^{T} \right]^{T}$. The iteration algorithm can be written as the following compact form
%By using the effective consensus protocol proposed in \cite{yi2019average}, we add parameter $\varepsilon$
%% to the communication weight matrix 
%and use $\left( I-\varepsilon \mathcal{L} \right) \otimes I_d $ as the communication matrix. We obtain the algorithm DIGing becomes
%\textcolor{blue}{Adding parameter} $\varepsilon$ to the communication weight matrix as the way proposed in \cite{yi2019average} and using $\left( I-\varepsilon \mathcal{L} \right) $ as the communication matrix, we obtain
%	The specialized algorithm DIGing is
\begin{equation}
	\label{eq: DIGing_with_epsilon}
	\begin{split}
		\boldsymbol{x}\left( k+1 \right) &=\left( I-\varepsilon \mathcal{L} \right)\otimes I_d \boldsymbol{x}\left( k \right) -\alpha \boldsymbol{y}\left( k \right),\\
		\boldsymbol{y}\left( k+1 \right) &=\left( I-\varepsilon \mathcal{L} \right)\otimes I_d \boldsymbol{y}\left( k \right) +\boldsymbol{x}\left( k+1 \right) -\boldsymbol{x}\left( k \right) .\\
	\end{split}
\end{equation}
Next, we define the worst-case convergence rate.

\begin{Def}
	\label{def: gamma}
	One algorithm is said to converge to the solution of problem \eqref{eq: problem average consensus} at a convergence rate $\gamma$ if 
	$$\underset{k\rightarrow \infty}{\lim}\rho ^{-k}\left\| x\left( k \right) -x^* \right\| \,\,=0,\, \forall \rho \in \left( \gamma ,1 \right].$$
\end{Def}

%\begin{Def}
We consider the worst-case convergence rate on a set of uncertain connected graphs.
Let $\left\{ \mathcal{G} \right\} _{\left[ \underline{\lambda}, \overline{\lambda} \right]}$ be the set of all connected graphs with $\left[ \lambda _2,\lambda _N \right] \subseteq \left[ \underline{\lambda},\overline{\lambda} \right] $, where $\lambda _2,\lambda _N$ are the smallest and largest positive eigenvalues of the Laplacian matrix $\mathcal{L}$. 
For the modified DIGing algorithm \eqref{eq: entry_DIG with epsilon}, the optimal parameter design problem is to find $\varepsilon $ and $\alpha $, so that the worst convergence rate is as fast as possible. This can be represented as the following optimization problem
\begin{equation*}
	\gamma ^*=\underset{\alpha ,\varepsilon}{\min}\underset{\mathcal{G} \in \left\{ \mathcal{G} \right\} _{\left[ \underline{\lambda},\overline{\lambda} \right]}}{sup}\gamma ,
%	\gamma ^{* }=\min _{\alpha, \epsilon} \sup _{G \in {G}} \gamma
\end{equation*}
where $\gamma$ is defined in Definition \ref{def: gamma}.
%The worst-case optimal convergence rate of algorithm DIGing is defined as
%	$$\gamma ^*=\underset{\mathcal{G} \in \left\{ \mathcal{G} \right\} _{\left[ \alpha ,\beta \right]}}{sup}\gamma, $$
%	where $\gamma$ is defined in Definition \ref{def: gamma}. 
In this paper, with no loss of generality, we set $\underline{\lambda} \triangleq \lambda _2, \overline{\lambda} \triangleq \lambda _N$ for notation simplicity.
%\end{Def}

%$\eta _i\triangleq 1-\varepsilon d_i$,
%\begin{equation}
%	\begin{split}
%		\left\{x_i\left( k+1 \right) =x_i\left( k \right) -\varepsilon \left[d_i x_i\left( k \right) -\sum_{j\in \mathcal{N} _i}{w_{ij}x_j\left( k \right)} \right] -\alpha y_i\left( k \right)\\
%		y_i\left( k+1 \right) =y_i\left( k \right)& -\varepsilon \left[d_i y_i\left( k \right) -\sum_{j\in \mathcal{N} _i}{w_{ij}y_j\left( k \right)} \right]\\ 
%		&+L\left[x_i\left( k+1 \right) -x_i\left( k \right)\right]
%		\right.
%	\end{split}
%\end{equation}
%\begin{gather}
%	\begin{cases}
%		x_i\left( k+1 \right) =x_i\left( k \right) -\varepsilon \left[d_i x_i\left( k \right) -\sum_{j\in \mathcal{N} _i}{w_{ij}x_j\left( k \right)} \right] -\alpha y_i\left( k \right)\\
%		y_i\left( k+1 \right) =y_i\left( k \right) -\varepsilon \left[d_i y_i\left( k \right) -\sum_{j\in \mathcal{N} _i}{w_{ij}y_j\left( k \right)} \right]\\ 
%		\,\,\,\,\,\,\,\,\,\,\,\,\,\,\,\,\,\,\,\,+L\left[x_i\left( k+1 \right) -x_i\left( k \right)\right]\\
%	\end{cases}.
%\end{gather}

%\textcolor{blue}{(Delete?)It equals to}
%\begin{gather}
%	\begin{cases}
%		x_i\left( k+1 \right) =\left( 1-\varepsilon d_i \right) x_i\left( k \right) +\varepsilon \sum_{j\in \mathcal{N} _i}{w_{ij}x_j\left( k \right)}-\alpha y_i\left( k \right)\\
%		y_i\left( k+1 \right) =\left( 1-\varepsilon d_i -L\alpha \right) y_i\left( k \right) -\varepsilon L d_i x_i\left( k \right) +\varepsilon \sum_{j\in \mathcal{N} _i}{w_{ij}\left[L x_j\left( k \right) +y_j\left( k \right) \right]}\\
%	\end{cases}.
%\end{gather}

\section{Convergence analysis and optimal parameter design}
%	用必要条件里的不等式，
\subsection{Problem Transformation and Decomposition}
%\subsection{Frequency Domain Transformation and Decomposition}
%\subsection{Problem Transformation}
%\subsection{A new approach for convergence analysis from the system perspective}
In this section, we propose a new approach for the convergence analysis of the algorithm \eqref{eq: entry_DIG with epsilon}.
%DIGing on minimizing unweighted sum of squares, and provide the optimal parameter design. 
We shall derive the explicit formula of the optimal worst-case convergence rate and the corresponding parameters.
We regard the algorithm iteration as a discrete dynamic system and give a decomposition 
%of the system 
from graph frequency domain.

%Consider algorithm DIGing formulated in \eqref{eq: DIGing_with_epsilon}. It 
The system \eqref{eq: entry_DIG with epsilon} can be written as the following discrete dynamic system 
\begin{equation}
	\label{eq: DIG system}
	\begin{split}
		\left[ \begin{array}{c}
			\boldsymbol{x}\left( k+1 \right)\\
			\boldsymbol{y}\left( k+1 \right)\\
		\end{array} \right] =&\left[ \begin{matrix}
		I-\varepsilon \mathcal{L}&		-\alpha I\\
		0&		I-\varepsilon \mathcal{L}\\
		\end{matrix} \right] \otimes I_d \left[ \begin{array}{c}
			\boldsymbol{x}\left( k \right)\\
			\boldsymbol{y}\left( k \right)\\
		\end{array} \right] \\&+\left[ \begin{array}{c}
			0\\
			-I\\
		\end{array} \right] \otimes I_d u\left( k \right),\\
		u\left( k \right) =&\left[ \begin{matrix}
			\varepsilon \mathcal{L}&		\alpha I\\
		\end{matrix} \right] \otimes I_d \left[ \begin{array}{c}
			\boldsymbol{x}\left( k \right)\\
			\boldsymbol{y}\left( k \right)\\
		\end{array} \right].
	\end{split}
\end{equation}
%\begin{equation}
%	\label{eq: DIG system}
%	\begin{cases}
%		\left[ \begin{array}{c}
%			\boldsymbol{x}\left( k+1 \right)\\
%			\boldsymbol{y}\left( k+1 \right)\\
%		\end{array} \right] =J_0 \left[ \begin{array}{c}
%			\boldsymbol{x}\left( k \right)\\
%			\boldsymbol{y}\left( k \right)\\
%		\end{array} \right] +\left[ \begin{array}{c}
%			0\\
%			-LI\\
%		\end{array} \right] u\left( k \right)\\
%		u\left( k \right) =\left[ \begin{matrix}
%			\varepsilon \mathcal{L}&		\alpha I\\
%		\end{matrix} \right] \left[ \begin{array}{c}
%			\boldsymbol{x}\left( k \right)\\
%			\boldsymbol{y}\left( k \right)\\
%		\end{array} \right]
%	\end{cases},
%\end{equation}
%where $J_0=\left[ \begin{matrix}
%	I-\varepsilon \mathcal{L}&		-\alpha I\\
%	0&		I-\varepsilon \mathcal{L}\\
%\end{matrix} \right]$.

Following the singular value decomposition of the Laplacian matrix $\mathcal{L} =U\varLambda U^T$ in Lemma \ref{lemma 1}, we give the unitary transformation of the state space. %give/introduce
The unitary coordinate transformation at iteration $k$ is formulated by 
\begin{equation*}
%	\label{equ unitary decomp}
	\tilde{\boldsymbol{x}}\left( k \right) =U^T\otimes I_d \boldsymbol{x}\left( k \right) ,\, \tilde{\boldsymbol{y}}\left( k \right) =U^T\otimes I_d \boldsymbol{y}\left( k \right), \, \forall  \boldsymbol{x},\boldsymbol{y} \in \mathbb{R}^{Nd} .
\end{equation*}
%	$$$$
Denote $\tilde{\boldsymbol{x}}=\left[ \left( \tilde{x}_1 \right) ^T,\left( \tilde{x}_2 \right) ^T,\cdots ,\left( \tilde{x}_N \right) ^T \right] ^T\in \mathbb{R} ^{Nd}$, where $\tilde{x}_i=\left( \tilde{x}_{i}^{1},\tilde{x}_{i}^{2},\cdots ,\tilde{x}_{i}^{d} \right) ^T\in \mathbb{R}^d,\,\, \forall i=1,\cdots,N$. 

Using singular value decomposition of Laplacian matrix as $\mathcal{L} =U\varLambda U^T$, with $\varLambda =\mathrm{diag} \{ 0,\lambda _2,\cdots ,\lambda _N \}  $, the system in \eqref{eq: DIG system} can be written as
\begin{equation}
	\label{eq: DIG system unitary}
	\begin{split}
		\left[ \begin{array}{c}
			\tilde{\boldsymbol{x}}\left( k+1 \right)\\
			\tilde{\boldsymbol{y}}\left( k+1 \right)\\
		\end{array} \right] =&\left[ \begin{matrix}
		I-\varepsilon \varLambda&		-\alpha I\\
		0&		I-\varepsilon \varLambda\\
		\end{matrix} \right] \otimes I_d \left[ \begin{array}{c}
			\tilde{\boldsymbol{x}}\left( k \right)\\
			\tilde{\boldsymbol{y}}\left( k \right)\\
		\end{array} \right] \\ &+\left[ \begin{array}{c}
			0\\
			-I\\
		\end{array} \right] \otimes I_d \tilde{u}\left( k \right),\\
		\tilde{u}\left( k \right) =&\left[ \begin{matrix}
			\varepsilon \varLambda&		\alpha I\\
		\end{matrix} \right] \otimes I_d \left[ \begin{array}{c}
			\tilde{\boldsymbol{x}}\left( k \right)\\
			\tilde{\boldsymbol{y}}\left( k \right)\\
		\end{array} \right].\\
	\end{split}
\end{equation}
%\begin{gather}
%	\label{eq: DIG system unitary}
%	\begin{cases}
%		\left[ \begin{array}{c}
%			\tilde{\boldsymbol{x}}\left( k+1 \right)\\
%			\tilde{\boldsymbol{y}}\left( k+1 \right)\\
%		\end{array} \right] =J^{\circ} \left[ \begin{array}{c}
%			\tilde{\boldsymbol{x}}\left( k \right)\\
%			\tilde{\boldsymbol{y}}\left( k \right)\\
%		\end{array} \right] +\left[ \begin{array}{c}
%			0\\
%			-LI\\
%		\end{array} \right]^{\circ} \tilde{u}\left( k \right)\\
%		\tilde{u}\left( k \right) =\left[ \begin{matrix}
%			\varepsilon \varLambda&		\alpha I\\
%		\end{matrix} \right]^{\circ} \left[ \begin{array}{c}
%			\tilde{\boldsymbol{x}}\left( k \right)\\
%			\tilde{\boldsymbol{y}}\left( k \right)\\
%		\end{array} \right]\\
%	\end{cases},
%\end{gather}
%where $J=\left[ \begin{matrix}
%	I-\varepsilon \varLambda&		-\alpha I\\
%	0&		I-\varepsilon \varLambda\\
%\end{matrix} \right]$.

Note that $\tilde{\boldsymbol{x}}\left( k+1 \right) \triangleq \left[ \begin{array}{c}
	\tilde{x}_1\left( k+1 \right)\\
	\tilde{\boldsymbol{x}}_{2:N}\left( k+1 \right)\\
\end{array} \right] $, with $\tilde{\boldsymbol{x}}_{2:N}\left( k+1 \right) =\left[ \left( \tilde{x}_2\left( k+1 \right) \right) ^T,\cdots ,\left( \tilde{x}_N\left( k+1 \right) \right) ^T \right] ^T$. Decompose $\tilde{x}_1\left( k \right) $ separately, we can write \eqref{eq: DIG system unitary} as
\begin{gather}
	\label{system-separate1}
	\tilde{x}_1\left( k+1 \right) =\tilde{x}_1\left( k \right) -\alpha \tilde{y}_1\left( k \right) ,
	\\
	\label{system-separate2}
	\left[ \begin{array}{c}
		\tilde{\boldsymbol{x}}_{2:N}\left( k+1 \right)\\
		\tilde{\boldsymbol{y}}\left( k+1 \right)\\
	\end{array} \right] =J \left[ \begin{array}{c}
		\tilde{\boldsymbol{x}}_{2:N}\left( k \right)\\
		\tilde{\boldsymbol{y}}\left( k \right)\\
	\end{array} \right] +\left[ \begin{array}{c}
		0_{\left( N-1 \right)d \times Nd}\\
		-I_{Nd}\\
	\end{array} \right] \tilde{u}\left( k \right) ,
	\\
	\label{system-separate3}
	\tilde{u}\left( k \right) =\left[ \begin{array}{c}
		0_{1\times \left( N-1 \right)}\\
		\varepsilon \varLambda _{2:N}\\
	\end{array} \middle| \alpha I_N \right] _{N\times \left( N-1+N \right)} \otimes I_d \left[ \begin{array}{c}
		\tilde{\boldsymbol{x}}_{2:N}\left( k \right)\\
		\tilde{\boldsymbol{y}}\left( k \right)\\
	\end{array} \right] ,
\end{gather}
where $J=\left[ \begin{matrix}
	I_{N-1}-\varepsilon \varLambda _{2:N}&		\left[ \begin{matrix}
		0_{\left( N-1 \right) \times 1}&		-\alpha I_{N-1}\\
	\end{matrix} \right]\\
	0_{N\times \left( N-1 \right)}&		I_N-\varepsilon \varLambda _N\\
\end{matrix} \right] \otimes I_d$, and
%.
%In \eqref{system-separate2}-\eqref{system-separate3}, 
$\varLambda _{2:N}=\mathrm{diag}\left\{ \lambda _2,\cdots ,\lambda _N \right\} $.
The system \eqref{system-separate1}-\eqref{system-separate3} is shown in Fig. \ref{fig:DIG system}.

\begin{figure}[h]  % DIGing 二次型函数 i=1
	\centering
	% 流程图定义基本形状
	\tikzstyle{process} = [rectangle, minimum width=1cm, minimum height=1cm, text centered, text width = 1cm, inner sep = 8pt, draw=black]
	\tikzstyle{process_thin} = [rectangle, minimum width=0.5cm, minimum height=0.8cm, text centered, text width = 0.5cm, inner sep = 8pt, draw=black]
	\tikzstyle{cycle} = [circle, minimum width=0.25cm, minimum height=0.25cm, text centered, inner sep = 1.5pt, draw=black, fill=white]
	% text width = 2.4cm, inner sep = 8pt 保证了可以用“\\”换行
	% 箭头形式
	\tikzstyle{arrow} = [->,>=stealth]
	\begin{tikzpicture}[node distance=0.4cm,
		arrow1/.style = {draw = black, line width=2pt, {Latex[length = 2mm, width = 2.5mm]}-{Latex[length = 2mm, width = 2.5mm]},},]
		%定义流程图具体形状
		\node(pro1)[process_thin, line width=1.2pt, yshift = -1cm]{$z^{-1}$};
		\node(pro2)[process_thin, line width=1.2pt, below of = pro1, yshift = -0.8cm]{$\alpha$};
		\node(pro3)[process, line width=1.2pt, below of = pro2, yshift = -1cm]{$G(z)$};
		\node(pro4)[process, line width=1.2pt, below of = pro3, yshift = -1cm]{$I_{Nd}$};
		\node(cir1)[cycle, line width=1pt, left of = pro1, xshift=-2.0cm]{};
		\node(cir2)[cycle, line width=1pt, left of = pro3, xshift=-2.5cm]{};
		\coordinate (point1) at (-2.4cm, -2.2cm); %α左
		\coordinate (point2) at (-3cm, -1cm);
		\coordinate (point10) at (-2.4cm, -1.2cm); %-号
		\coordinate (point3) at (2.1cm, -1cm); %z逆右
		\coordinate (point4) at (2.1cm, 0.02cm); %最右上角
		\coordinate (point8) at (-2.4cm, 0.02cm); %最左上角
		\coordinate (point9) at (-2.4cm, -0.8cm); %+号
		\coordinate (point5) at (1.8cm, -2.2cm); %α右
		\coordinate (point6) at (1.8cm, -3.4cm);
		\coordinate (point7) at (0.78cm, -3.4cm);
		\coordinate (point11) at (3.2cm, -3.6cm); %G(z)到u(k)
		\coordinate (point12) at (2.7cm, -3.61cm); %输出u(k)的位置
		\coordinate (point13) at (2.7cm, -5cm); %LI右
		\coordinate (point14) at (-2.9cm, -5cm); %LI左
		\coordinate (point15) at (-2.9cm, -4.3cm); %下面的+号
		%连接具体形状
		\draw [very thick](pro2) -- node [left] {} (point1);
		\draw [very thick][arrow](point1) -- (cir1);
		\draw [very thick](point1) |- node[below]{$-\,\,\,\,\,\,\,\,$}(point10); %-号
		\draw [very thick][arrow](cir1) --node[above]{$\tilde{x}_1\left( k+1 \right) $}(pro1);
		\draw [very thick](pro1) -- node [above] {$\tilde{x}_1\left( k \right) $} (point3);
		\draw [very thick](point3) -- (point4);
		\draw [very thick](point4) -- node [right] {} (point8);
		\draw [very thick](point8) |- node[above]{$+\,\,\,\,\,\,\,\,$}(point9); %+号
		\draw [very thick][arrow](point8) -- (cir1);
		\draw [very thick][arrow](point5)--(pro2);
		\draw [very thick](point6)--node[right]{$\tilde{y}_1\left( k \right) $}(point5);
		\draw [very thick][arrow](point7) -- (point6);
		\draw [very thick][arrow](pro3) -- (point11);
		\draw [very thick](point13) |- node[above]{$\,\,\,\,\tilde{u}(k)$}(point12);
		\draw [very thick][arrow](point13) -- (pro4);
		\draw [very thick](pro4) -- (point14);
		\draw [very thick][arrow](point14) -- (cir2);
		\draw [very thick][arrow](cir2) -- node[above]{$\tilde{u}(k)$}(pro3);
		\draw [very thick](point14) |- node[above]{$+\,\,\,\,\,\,\,\,$}(point15); %+号
	\end{tikzpicture}
	\caption{The whole system after unitary transformation} %for solving average consensus
	%	The whole discrete dynamic system of DIGing after unitary transformation
	\label{fig:DIG system}
\end{figure}
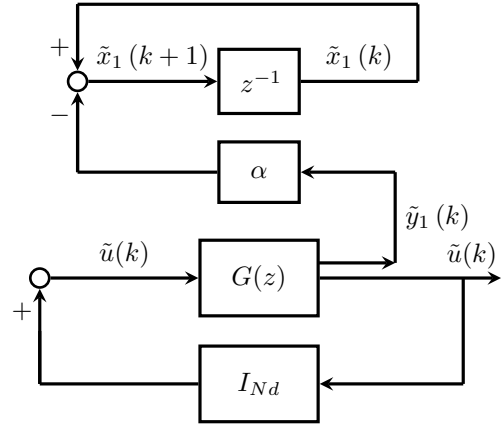

%Next, we investigate the \textcolor{black}{worst-case optimal convergence rate} of system \eqref{system-separate1}-\eqref{system-separate3}. 
%/stability

Now consider only the subsystem of \eqref{system-separate2}-\eqref{system-separate3}.
The transfer matrix $G(z)$ can be computed as in \eqref{eq: G(z)} (at the bottom of next page), where $\varDelta _1\left( z \right) =zI_{N-1}-I_{N-1}+\varepsilon \varLambda _{2:N}$, $\varDelta _2\left( z \right) =zI_N-I_N+\varepsilon \varLambda _N$.
%only
%\begin{multicols}{1}
%	\begin{equation}
%		
%	\end{equation}
%\end{multicols}
%where 
Then, 
\begin{equation*}
	\begin{split}
		G\left( z \right) =&-\alpha \left( zI_N-I_N+\varepsilon \varLambda _N \right) ^{-1} \otimes I_d\\
		&+\left[ \begin{matrix}
			0&		0\\
			0&		\alpha \varepsilon \varLambda _{2:N}\left( zI_{N-1}-I_{N-1}+\varepsilon \varLambda _{2:N} \right) ^{-2}\\
		\end{matrix} \right] \otimes I_d
		\\
		=&\left[ \begin{matrix}
			\frac{-\alpha}{z-1}&		&		&		\\
			&		\frac{-\alpha \left( z-1 \right)}{\left( z-1+\varepsilon \lambda _2 \right) ^2}&		&		\\
			&		&		\ddots&		\\
			&		&		&		\frac{-\alpha \left( z-1 \right)}{\left( z-1+\varepsilon \lambda _N \right) ^2}\\
		\end{matrix} \right] \otimes I_d.
	\end{split}
\end{equation*}
So far, the problem becomes to find $\alpha$, $\varepsilon$, so that the subsystem \eqref{system-separate2}-\eqref{system-separate3} converges to \textcolor{black}{zero} as fast as possible, \textcolor{black}{for all networks} with $\lambda _i\in \left[ \lambda _2,\lambda _N \right],\,i=2,\cdots,N$. 
%\textcolor{blue}{Note that with Assumption} \ref{ass graph} holds, we have $0<\lambda _2\leqslant \lambda _N$.
The subsystem \eqref{system-separate2}-\eqref{system-separate3} is shown in Fig. \ref{fig:DIG system2}.

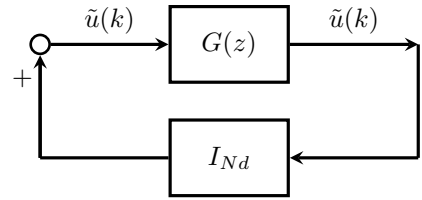
\begin{figure}[h]  % DIGing 二次型函数 i=1
	\centering
	% 流程图定义基本形状
	\tikzstyle{process} = [rectangle, minimum width=1cm, minimum height=1cm, text centered, text width = 1cm, inner sep = 8pt, draw=black]
	\tikzstyle{process_thin} = [rectangle, minimum width=1cm, minimum height=1cm, text centered, text width = 1cm, inner sep = 8pt, draw=black]
	\tikzstyle{cycle} = [circle, minimum width=0.25cm, minimum height=0.25cm, text centered, inner sep = 1.5pt, draw=black, fill=white]
	% text width = 2.4cm, inner sep = 8pt 保证了可以用“\\”换行
	% 箭头形式
	\tikzstyle{arrow} = [->,>=stealth]
	\begin{tikzpicture}[node distance=0.5cm]
		%定义流程图具体形状
		\node(pro1)[process_thin, line width=1.2pt, yshift = -1cm]{$G(z)$};
		\node(pro2)[process, line width=1.2pt, below of = pro1, yshift = -1cm]{$I_{Nd}$};
		%			\node(pro3)[process, below of = pro2, yshift = -1cm]{$G(z)$};
		%			\node(pro4)[process, below of = pro3, yshift = -1cm]{$\varSigma _Q$};
		\node(cir1)[cycle, line width=1pt, left of = pro1, xshift=-2cm]{};
		%			\node(cir2)[cycle, left of = pro3, xshift=-2.5cm]{};
		\coordinate (point1) at (-2.5cm, -2.5cm);
		\coordinate (point2) at (-2.5cm, -1cm);
		\coordinate (point10) at (-2.5cm, -1.2cm); %+号
		\coordinate (point3) at (2.5cm, -1cm);
		%			\coordinate (point4) at (2.5cm, 0.5cm);
		%			\coordinate (point8) at (-3cm, 0.5cm);
		%			\coordinate (point9) at (-3cm, -0.85cm); %+号
		\coordinate (point5) at (2.5cm, -2.5cm);
		%			\coordinate (point6) at (2cm, -3.7cm);
		%			\coordinate (point7) at (0.78cm, -3.7cm);
		%			\coordinate (point11) at (3.5cm, -4cm);
		%			\coordinate (point12) at (3cm, -4cm);
		%			\coordinate (point13) at (3cm, -5.5cm);
		%			\coordinate (point14) at (-3cm, -5.5cm);
		%			\coordinate (point15) at (-3cm, -4.65cm);
		%连接具体形状
		\draw [very thick](pro2) -- node [left] {} (point1);
		\draw [very thick][arrow](point1) -- (cir1);
		\draw [very thick](point1) |- node[below]{$+\,\,\,\,\,\,\,\,$}(point10); %-号
		\draw [very thick][arrow](cir1) --node[above]{$\tilde{u}(k)$}(pro1);
		\draw [very thick][arrow](pro1) -- node [above] {$\tilde{u}(k)$} (point3);
		%			\draw (point3) -- (point4);
		%			\draw (point4) -- node [right] {} (point8);
		%			\draw (point8) |- node[above]{$+\,\,\,\,\,\,\,\,$}(point9); %+号
		%			\draw [arrow](point8) -- (cir1);
		\draw [very thick][arrow](point5)--(pro2);
		\draw [very thick](point3)--(point5);
		%			\draw [arrow](point7) -- (point6);
		%			\draw [arrow](pro3) -- (point11);
		%			\draw (point13) |- node[above]{$\,\,\,\,v(k)$}(point12);
		%			\draw [arrow](point13) -- (pro4);
		%			\draw (pro4) -- (point14);
		%			\draw [arrow](point14) -- (cir2);
		%			\draw [arrow](cir2) -- node[above]{$u(k)$}(pro3);
		%			\draw (point14) |- node[above]{$+\,\,\,\,\,\,\,\,$}(point15); %+号
	\end{tikzpicture}
	\caption{The second subsystem after decomposition} %The second subsystem of DIGing after graph frequency domain decomposition
	\label{fig:DIG system2}
\end{figure}

\begin{figure*}[hb]
	\begin{equation}
		\label{eq: G(z)}
		\begin{split}
			G\left( z \right) &=\left[ \begin{array}{c}
				0_{1\times \left( N-1 \right)}\\
				\varepsilon \varLambda _{2:N}\\
			\end{array} \middle| \alpha I_N \right] \left[ \begin{matrix}
				zI_{N-1}-I_{N-1}+\varepsilon \varLambda _{2:N}&		\left[ \begin{matrix}
					0_{\left( N-1 \right) \times 1}&		\alpha I_{N-1}\\
				\end{matrix} \right]\\
				0_{N\times \left( N-1 \right)}&		zI_N-I_N+\varepsilon \varLambda _N\\
			\end{matrix} \right] ^{-1}\left[ \begin{array}{c}
				0_{\left( N-1 \right) \times N}\\
				-I_N\\
			\end{array} \right]  \otimes I_d
			\\
			&=\left[ \begin{array}{c}
				0_{1\times \left( N-1 \right)}\\
				\varepsilon \varLambda _{2:N}\\
			\end{array} \middle| \alpha I_N \right] \left[ \begin{matrix}
				\varDelta _{1}^{-1}\left( z \right)&		-\varDelta _{1}^{-1}\left( z \right) \left[ \begin{matrix}
					0_{\left( N-1 \right) \times 1}&		\alpha I_{N-1}\\
				\end{matrix} \right] \varDelta _{2}^{-1}\left( z \right)\\
				0_{N\times \left( N-1 \right)}&		\varDelta _{2}^{-1}\left( z \right)\\
			\end{matrix} \right] \left[ \begin{array}{c}
				0_{\left( N-1 \right) \times N}\\
				-I_N\\
			\end{array} \right]  \otimes I_d
%			\\
%			&=\left[ \begin{array}{c}
%				0_{1\times \left( N-1 \right)}\\
%				\varepsilon \varLambda _{2:N}\\
%			\end{array} \middle| \alpha I_N \right] \left[ \begin{matrix}
%				\varDelta _{1}^{-1}\left( z \right)&		-\left[ \begin{matrix}
%					0_{\left( N-1 \right) \times 1}&		\alpha \varDelta _{1}^{-2}\left( z \right)\\
%				\end{matrix} \right]\\
%				0_{N\times \left( N-1 \right)}&		\varDelta _{2}^{-1}\left( z \right)\\
%			\end{matrix} \right] \left[ \begin{array}{c}
%				0_{\left( N-1 \right) \times N}\\
%				-I_N\\
%			\end{array} \right]  \otimes I_d
			\\
			&=\left[ \begin{array}{c}
				0_{1\times \left( N-1 \right)}\\
				\varepsilon \varLambda _{2:N}\\
			\end{array} \middle| \alpha I_N \right] \left[ \begin{array}{c}
				\left[ \begin{matrix}
					0_{\left( N-1 \right) \times 1}&		\alpha \varDelta _{1}^{-2}\left( z \right)\\
				\end{matrix} \right]\\
				-\varDelta _{2}^{-1}\left( z \right)\\
			\end{array} \right] \otimes I_d =-\alpha \varDelta _{2}^{-1}\left( z \right) \otimes I_d +\left[ \begin{matrix}
				0&		0\\
				0&		\alpha \varepsilon \varLambda _{2:N}\varDelta _{1}^{-2}\left( z \right)\\
			\end{matrix} \right]  \otimes I_d.
		\end{split}
	\end{equation}
\end{figure*}

The problem is transformed to finding the minimum of parameter $\gamma$ such that the closed-loop system in Fig. \ref{fig:DIG system2} with $G(\gamma z)$ is stable. In the following, we give the necessary and sufficient condition for the stable system.

The closed-loop system with $G(\gamma z)$ and $I_{Nd}$ can be computed as $\left[ \begin{matrix}
	\varPhi _1\left( \gamma z \right)&		&		\\
	&		\ddots&		\\
	&		&		\varPhi _N\left( \gamma z \right)\\
\end{matrix} \right] \otimes I_d $. Here, 
$$\varPhi _1\left( \gamma z \right) =\frac{-\frac{\alpha}{\gamma z-1}}{1+\frac{\alpha}{\gamma z-1}}=\frac{-\alpha}{\gamma z-1+\alpha},$$
$$\varPhi _i\left( \gamma z \right) =\frac{-\frac{\alpha \left( \gamma z-1 \right)}{\left( \gamma z-1+\varepsilon \lambda _i \right) ^2}}{1+\frac{\alpha \left( \gamma z-1 \right)}{\left( \gamma z-1+\varepsilon \lambda _i \right) ^2}}=\frac{-\alpha \left( \gamma z-1 \right)}{\left( \gamma z-1+\varepsilon \lambda _i \right) ^2+\alpha \left( \gamma z-1 \right)}.$$
The problem becomes finding the smallest $\gamma$ so that there exist $\varepsilon$ and $\alpha$ making all $\varPhi _i\left( \gamma z \right) $s stable. It is formulated as
\begin{equation}
	\label{eq-problem}
	\begin{gathered}
		\underset{\varepsilon ,\alpha}{\min}\,\,\gamma 
		\\
		s.t.\, \varPhi _i\left( \gamma z \right) \,\,stable.
%		\, i.e.\, all\,\,poles\,\,of\,\,\varPhi _i\left( \gamma z \right) \,\,are\,\,within\,\,the\,\,unit\,\,circle.
	\end{gathered}
\end{equation}
Note that $\varPhi _i\left( \gamma z \right)$ is stable if and only if all poles of $\varPhi _i\left( \gamma z \right)$ are within the unit circle. Then, $\varPhi _1\left( \gamma z \right) $ stable means
$-\gamma \leqslant 1-\alpha \leqslant \gamma $.
For $\varPhi _i\left( \gamma z \right) $ with $i=2,\cdots ,N$,
setting $z=\frac{s+1}{s-1}$.
The poles of $\varPhi _i\left( \gamma z \right) $ lie in the unit circle \textit{iff} the poles of $\varPhi _i\left( s \right) $ lie in left half plane.
It \textcolor{black}{is equivalent} to the roots of $$\left( \gamma \frac{s+1}{s-1}-1 \right) ^2+\left( 2\varepsilon \lambda _i+\alpha \right) \left( \gamma \frac{s+1}{s-1}-1 \right) +\varepsilon ^2\lambda _{i}^{2}=0$$ lie in left half plane.
%It equals to
%$$\gamma ^2\left( \frac{s+1}{s-1} \right) ^2-2\gamma \frac{s+1}{s-1}+1+\left( 2\varepsilon \lambda _i+L\alpha \right) \left( \gamma \frac{s+1}{s-1}-1 \right) +\varepsilon ^2\lambda _{i}^{2}=0.$$ It equals to
%$$\gamma ^2\left( s+1 \right) ^2-2\gamma \left( s+1 \right) \left( s-1 \right) +\left( s-1 \right) ^2+\left( 2\varepsilon \lambda _i+L\alpha \right) \left[ \gamma \left( s+1 \right) \left( s-1 \right) -\left( s-1 \right) ^2 \right] +\varepsilon ^2\lambda _{i}^{2}\left( s-1 \right) ^2=0.$$
Rewrite the equation as the quadratic polynomial of $s$, we obtain the following characteristic equation
%the formula \eqref{eq-s_poly} as
\begin{equation}
	\label{eq-s_poly}
	\begin{split}
		\left[ \left( \gamma -1 \right) ^2+\varepsilon ^2\lambda _{i}^{2}+\left( \gamma -1 \right) \left( 2\varepsilon \lambda _i+\alpha \right) \right] s^2+
		\\
		2\left[ \gamma ^2-1+2\varepsilon \lambda _i+\alpha -\varepsilon ^2\lambda _{i}^{2} \right] s+
		\\
		\left( \gamma +1 \right) ^2+\varepsilon ^2\lambda _{i}^{2}-\left( \gamma +1 \right) \left( 2\varepsilon \lambda _i+\alpha \right) =0.
	\end{split}
\end{equation}
Using \textit{Routh's Stability Criterion} in \cite{nise2019control}, we know that the roots of \eqref{eq-s_poly} lie in left half plane if and only if all the coefficients in \eqref{eq-s_poly} are non-negative. That is,
\begin{equation*}
	\begin{split}
		\begin{cases}
			\left( \gamma -1 \right) ^2+\varepsilon ^2\lambda _{i}^{2}+\left( \gamma -1 \right) \left( 2\varepsilon \lambda _i+\alpha \right) \geqslant 0\\
			\gamma ^2-1+2\varepsilon \lambda _i+\alpha -\varepsilon ^2\lambda _{i}^{2}\geqslant 0\\
			\left( \gamma +1 \right) ^2+\varepsilon ^2\lambda _{i}^{2}-\left( \gamma +1 \right) \left( 2\varepsilon \lambda _i+\alpha \right) \geqslant 0\\
		\end{cases},
	\end{split}
\end{equation*}
holds for all $\lambda _i\in \left[ \lambda _2,\lambda _N \right] $.
%$\forall \lambda _i\in \left\{ \lambda _2,\cdots ,\lambda _N \right\} $. 

Then, the optimization problem of \eqref{eq-problem} becomes to
\begin{equation}
	\label{eq: solve the eq}
	\underset{\alpha>0 ,\varepsilon>0}{\min}\,\,\gamma 
\end{equation}
$s.t.$
\begin{subnumcases}{}
	\alpha \leqslant 1+\gamma, \label{eq: solve sub-1} \\
	\alpha \geqslant 1-\gamma, \label{eq: solve sub-2} \\
	\left( \varepsilon \lambda _i-1+\gamma \right) ^2-\alpha \left( 1-\gamma \right) \geqslant 0, \label{eq: solve sub-3}\\ 
	\left( \varepsilon \lambda _i-1-\gamma \right) ^2-\alpha \left( 1+\gamma \right) \geqslant 0, \label{eq: solve sub-4} \\
	\gamma ^2-\left( 1-\varepsilon \lambda _i \right) ^2+\alpha \geqslant 0 \label{eq: solve sub-5},
\end{subnumcases}
%\begin{subequations}
%	\begin{align}
%%		\begin{cases}
%			L\alpha \leqslant 1+\gamma \\
%%			\,\,\,\,\,\,       \textcircled{1}\\
%			L\alpha \geqslant 1-\gamma \\
%%			\,\,\,\,\,\,       \textcircled{2}\\
%			\left( \varepsilon \lambda _i-1+\gamma \right) ^2-L\alpha \left( 1-\gamma \right) \geqslant 0\\ 
%%\,\,\,\,\,\,      \textcircled{3}\\
%			\left( \varepsilon \lambda _i-1-\gamma \right) ^2-L\alpha \left( 1+\gamma \right) \geqslant 0 \\
%%			 \,\,\,\,\,\,     \textcircled{4}\\
%			\gamma ^2-\left( 1-\varepsilon \lambda _i \right) ^2+L\alpha \geqslant 0   
%%			\,\,\,\,\,\,     \textcircled{5}\\
%%		\end{cases},
%	\end{align}
%\end{subequations}
where inequalities \eqref{eq: solve sub-3}-\eqref{eq: solve sub-5} hold for all $\lambda _i\in \left[ \lambda _2,\lambda _N \right] $.
%\textcolor{blue}{Note that due to $\gamma$ being the convergence rate of the system, it satisfies $0<\gamma \leqslant 1$.}

\subsection{Optimal Parameter Design}
%\subsection{Problem Solving and Parameter Design}
In this subsection, we will solve the minimization problem \eqref{eq: solve the eq} with constraints \eqref{eq: solve sub-1}-\eqref{eq: solve sub-5}, which gives the optimal convergence rate $\gamma ^*$ and the corresponding parameters $\alpha^* $ and $\varepsilon^* $.
%solve the transformed linear optimization problem subject to several quadratic constraints and derive the optimal parameter design for DIGing.
%Consider the problem in \eqref{eq: solve the eq}.
First, we present two lemmas that will be used to derive the main result.
%, which give the upper bound and lower bound for parameter $\varepsilon$.
%	/the following two lemmas hold.

%Next, we derive the lower bound for parameter $\varepsilon$.
\begin{Lem}
	\label{lem: epsilon lambda_2}
	Let $\gamma \leq 1$. Assume $\varepsilon >0$ and $\alpha >0$ satisfy inequalities \eqref{eq: solve sub-1}-\eqref{eq: solve sub-5} for all $\lambda _i \in [\lambda _2, \lambda _N]$.
%	Assume Assumption \ref{ass graph} holds. Consider the problem in \eqref{eq: solve the eq} subject to inequalities \eqref{eq: solve sub-1}-\eqref{eq: solve sub-5}.
%	Considering problem \eqref{eq: solve the eq}. 
	Then, we have $\varepsilon \lambda _2\geqslant 2\left( 1-\gamma \right) $ holds.
\end{Lem}
\begin{pf}
	From \eqref{eq: solve sub-2} and \eqref{eq: solve sub-3}, we obtain
	\begin{equation}
		\label{eq: lem 3}
		\left( \varepsilon \lambda _i-1+\gamma \right) ^2\geqslant \alpha \left( 1-\gamma \right) \geqslant \left( 1-\gamma \right) ^2.
	\end{equation}
	Then, $\varepsilon ^2\lambda _{i}^{2}-2\left( 1-\gamma \right) \varepsilon \lambda _i+\left( 1-\gamma \right) ^2\geqslant \left( 1-\gamma \right) ^2.$
	
	Then, we have $\varepsilon \lambda _i\left( \varepsilon \lambda _i-2+2\gamma \right) \geqslant 0,$
	which holds for $\forall \lambda _i\in \left[ \lambda _2,\lambda _N \right] $.
	Then, $\varepsilon \lambda _2-2+2\gamma \geqslant 0$, which is equivalent to $\varepsilon \lambda _2\geqslant 2\left( 1-\gamma \right) $.
\end{pf}

Similarly, we can get the following lemma.
\begin{Lem}
	\label{lem: epsilon lambda_N}
	Let $\gamma \leq 1$. Assume $\varepsilon >0$ and $\alpha >0$ satisfy inequalities \eqref{eq: solve sub-1}-\eqref{eq: solve sub-5} for all $\lambda _i \in [\lambda _2, \lambda _N]$.
	%	Assume Assumption \ref{ass graph} holds. Consider the problem in \eqref{eq: solve the eq} subject to inequalities \eqref{eq: solve sub-1}-\eqref{eq: solve sub-5}. 
	Then, we have $\varepsilon \lambda _N\leqslant 1+\gamma -\sqrt{1-\gamma ^2}$ holds.
\end{Lem}
%\begin{pf}
%	From \eqref{eq: solve sub-2} and \eqref{eq: solve sub-4}, we obtain
%	\begin{equation}
%		\label{eq: lem2}
%		\left( \varepsilon \lambda _i-1-\gamma \right) ^2\geqslant \alpha \left( 1+\gamma \right) \geqslant \left( 1-\gamma \right) \left( 1+\gamma \right) =1-\gamma ^2.
%	\end{equation}
%	Next, we use the method of proof by contradiction. Assume $\varepsilon \lambda _N>1+\gamma $ holds, then there exists $\lambda _j<\lambda _N$, such that $\varepsilon \lambda _j-1-\gamma =0$. Then, we have $0\geqslant 1-\gamma ^2$, which does not hold for $\gamma <1$. Contradiction.\\
%	Then, $\varepsilon \lambda _N<1+\gamma $ holds.
%	By \eqref{eq: lem2}, we have $$\left( 1+\gamma -\varepsilon \lambda _N \right) ^2\geqslant 1-\gamma ^2.$$ Then, $1+\gamma -\varepsilon \lambda _N\geqslant \sqrt{1-\gamma ^2}.$ \\
%	This is equivalent to $\varepsilon \lambda _N\leqslant 1+\gamma -\sqrt{1-\gamma ^2}$.
%\end{pf}

Following Lemma \ref{lem: epsilon lambda_2} and Lemma \ref{lem: epsilon lambda_N}, we can prove Theorem \ref{th1}, which provides the optimal convergence rate of system \eqref{system-separate2}-\eqref{system-separate3} and the corresponding parameters.
%give the proof of Theorem 1, which provides the optimal parameter design and the optimal convergence rate of the system \eqref{eq: system}.

\begin{thm}
	\label{th1}
	Suppose Assumption \ref{ass graph} holds. Consider the system in \eqref{system-separate2}-\eqref{system-separate3} \textcolor{black}{with $\lambda _i\in \left[ \lambda _2,\lambda _N \right],\,i=2,\cdots,N$}.
%	, where $\lambda_i$'s are eigenvalues of the Laplacian matrix.
	Let $\kappa \triangleq \frac{\lambda _2}{\lambda _N} \in \left( 0,1 \right]$.
%	\textcolor{black}{, then} $\kappa  $. 
	Then, the worst-case optimal convergence rate of system \eqref{system-separate2}-\eqref{system-separate3} is $\gamma ^*=\frac{4-\kappa ^2+\kappa \sqrt{\kappa ^2+8\kappa}}{2\left( \kappa ^2+2\kappa +2 \right)}$, and the corresponding parameters are 
	\begin{equation}
		\label{eq: opt parameters}
		\alpha ^*=1-\gamma ^*,\,\, \varepsilon ^*=\frac{2}{\lambda _2}\left( 1-\gamma ^* \right).
	\end{equation}
\end{thm}
\begin{pf}
	Following Lemma \ref{lem: epsilon lambda_2} and Lemma \ref{lem: epsilon lambda_N}, we have
	\begin{equation}
		\label{eq: epsilon}
		\frac{2}{\lambda _2}\left( 1-\gamma \right) \leqslant \varepsilon \leqslant \frac{1}{\lambda _N}\left( 1+\gamma -\sqrt{1-\gamma ^2} \right) .
	\end{equation}
	Notice that the value on the left side of inequality \eqref{eq: epsilon} increases as $\gamma$ decreases, while the value on the right side of inequality \eqref{eq: epsilon} decreases as $\gamma$ decreases. Thus, when both sides of inequality \eqref{eq: epsilon} are equal, $\gamma$ attains the minimum value.
	Taking the equality, we obtain $$\frac{2}{\lambda _2}\left( 1-\gamma \right) =\frac{1}{\lambda _N}\left( 1+\gamma -\sqrt{1-\gamma ^2} \right) .$$
	It is equivalent to
	\begin{equation}
		\label{eq: abandon the root}
		2\lambda _N-\lambda _2-\left( 2\lambda _N+\lambda _2 \right) \gamma =-\lambda _2\sqrt{1-\gamma ^2}.
	\end{equation}
	Squaring both sides of the equation \eqref{eq: abandon the root} yields
	$$\gamma ^2\left( \lambda _{2}^{2}+2\lambda _2\lambda _N+2\lambda _{N}^{2} \right) -\left( 4\lambda _{N}^{2}-\lambda _{2}^{2} \right) \gamma +2\lambda _{N}^{2}-2\lambda _2\lambda _N=0.$$
	The roots of the above equation are
	$$\gamma _1=\frac{4-\kappa ^2+\kappa \sqrt{\kappa ^2+8\kappa}}{2\left( \kappa ^2+2\kappa +2 \right)},\,\, \gamma _2=\frac{4-\kappa ^2-\kappa \sqrt{\kappa ^2+8\kappa}}{2\left( \kappa ^2+2\kappa +2 \right)}.$$
	
	Notice that substituting $\gamma_2$ into \eqref{eq: abandon the root} , the equation \eqref{eq: abandon the root} does not hold.
	Thus, the root $\gamma_2$ is discarded.
	Also, it is easy to verify that $0<\gamma _1<1$.
	Then, the optimal convergence rate is $\gamma ^*=\frac{4-\kappa ^2+\kappa \sqrt{\kappa ^2+8\kappa}}{2\left( \kappa ^2+2\kappa +2 \right)}$.
	And $\varepsilon ^*=\frac{2}{\lambda _2}\left( 1-\gamma ^* \right) $ follows from \eqref{eq: epsilon} directly.
	
	Since \eqref{eq: lem 3} holds for $\forall \lambda _i\in \left[ \lambda _2,\lambda _N \right] $, we have $$\left( \varepsilon ^*\lambda _2-1+\gamma ^{*} \right) ^2=\left( 1-\gamma ^* \right) ^2.$$
	Then, 
	\begin{gather*}
		\left( \varepsilon ^*\lambda _2-1+\gamma ^* \right) ^2-\alpha ^*\left( 1-\gamma ^* \right) \\=\left( 1-\gamma ^* \right) ^2-\alpha ^*\left( 1-\gamma ^* \right) \geqslant 0,
	\end{gather*}
	which gives $\left( 1-\gamma ^* \right) \left( 1-\gamma ^*-\alpha ^* \right) \geqslant 0.$ Thus, $1-\gamma ^*\geqslant \alpha ^*$.
	It follows from \eqref{eq: solve sub-2} that
%	Also, from \eqref{eq: solve sub-2}, we know
	$\alpha ^*\geqslant 1-\gamma ^*$.
	Then, we have $\alpha ^*=1-\gamma ^*$. 
%	Then, $\alpha ^*=\frac{1-\gamma ^*}{L}$.

	Next, we verify inequalities \eqref{eq: solve sub-1}-\eqref{eq: solve sub-5} hold for $\gamma ^*$, $\varepsilon ^*$ and $\alpha ^*$.
	Note that $\alpha ^*=1-\gamma ^*\leqslant 1+\gamma ^*$, inequalities \eqref{eq: solve sub-1} and \eqref{eq: solve sub-2} obviously hold.
	
	Substituting $\alpha ^*=1-\gamma ^*$ into \eqref{eq: solve sub-3}, we have 
	\begin{equation}
		\label{eq: verify 3}
		\left( \varepsilon ^*\lambda _i-1+\gamma ^* \right) ^2\geqslant \left( 1-\gamma ^* \right) ^2.
	\end{equation} 
	It follows from $\varepsilon ^*\lambda _i\geqslant \varepsilon ^*\lambda _2=2\left( 1-\gamma ^* \right) $ that $\varepsilon ^*\lambda _i-1+\gamma ^*\geqslant \varepsilon ^*\lambda _2-\left( 1-\gamma ^* \right) =1-\gamma ^*$.
%	 Then, \eqref{eq: verify 3} holds. 
	Then, the validity of \eqref{eq: solve sub-3} is finished.
	
	Substituting $\alpha ^*=1-\gamma ^*$ into the left side of inequality \eqref{eq: solve sub-4}, we have
	\begin{equation*}
		\begin{gathered}
			\left( \varepsilon ^*\lambda _i-1-\gamma ^* \right) ^2-\alpha ^*\left( 1+\gamma ^* \right)\\ 
			=\left( 1+\gamma ^*-\varepsilon ^*\lambda _i \right) ^2-1+\left( \gamma ^* \right) ^2.
		\end{gathered}
	\end{equation*}
	Due to $\varepsilon ^*\lambda _i\leqslant \varepsilon ^*\lambda _N<1+\gamma ^*$, we have $$1+\gamma ^*-\varepsilon ^*\lambda _i\geqslant 1+\gamma ^*-\varepsilon ^*\lambda _N>0.$$
	Then, $$\left( 1+\gamma ^*-\varepsilon ^*\lambda _i \right) ^2\geqslant \left( 1+\gamma ^*-\varepsilon ^*\lambda _N \right) ^2.$$
	From Lemma \ref{lem: epsilon lambda_N}, we know $\varepsilon ^*\lambda _N\leqslant 1+\gamma ^*-\sqrt{1-\left( \gamma ^* \right) ^2}$. It equals to $1+\gamma ^*-\varepsilon ^*\lambda _N\geqslant \sqrt{1-\left( \gamma ^* \right) ^2}\geqslant 0$.
	Then, $$\left( 1+\gamma ^*-\varepsilon ^*\lambda _N \right) ^2\geqslant 1-\left( \gamma ^* \right) ^2.$$
	Then, we have 
	\begin{gather*}
		\left( 1+\gamma ^*-\varepsilon ^*\lambda _i \right) ^2-1+\left( \gamma ^* \right) ^2\geqslant \\ \left( 1+\gamma ^*-\varepsilon ^*\lambda _N \right) ^2-1+\left( \gamma ^* \right) ^2\geqslant 0,
	\end{gather*}
	which finishes the validation
%	The validity 
	of inequality \eqref{eq: solve sub-4}. % is verified.
	
	Substituting $\alpha ^*=1-\gamma ^*$ into the left side of inequality \eqref{eq: solve sub-5}, we have
	$$\left( \gamma ^* \right) ^2-\left( 1-\varepsilon ^*\lambda _i \right) ^2+\alpha ^*=\left( \gamma ^* \right) ^2-\left( 1-\varepsilon ^*\lambda _i \right) ^2+1-\gamma ^*.$$
	In the first case, consider $\lambda_i$ satisfying $\varepsilon ^*\lambda _i-1\geqslant 0$. Then, $$0\leqslant \varepsilon ^*\lambda _i-1\leqslant \varepsilon ^*\lambda _N-1.$$
	Due to $\varepsilon ^*\lambda _N\leqslant 1+\gamma ^*$, we have $\varepsilon ^*\lambda _N-1\leqslant \gamma ^*$. Then, $$\left( \varepsilon ^*\lambda _i-1 \right) ^2\leqslant \left( \varepsilon ^*\lambda _N-1 \right) ^2\leqslant \left( \gamma ^* \right) ^2.$$ 
	Also, due to $1-\gamma ^*\geqslant 0$, we have 
	\begin{gather*}
		\left( \gamma ^* \right) ^2-\left( 1-\varepsilon ^*\lambda _i \right) ^2+\alpha ^*=\\ \left( \gamma ^* \right) ^2-\left( 1-\varepsilon ^*\lambda _i \right) ^2+1-\gamma ^*\geqslant 0.
	\end{gather*}
%	The validity of 
	Hence inequality \eqref{eq: solve sub-5} is verified.
	
	In the second case, consider $\lambda_i$ satisfying $\varepsilon ^*\lambda _i-1<0$. Then, $$\varepsilon ^*\lambda _2-1\leqslant \varepsilon ^*\lambda _i-1<0.$$ It equals to $1-\varepsilon ^*\lambda _2\geqslant 1-\varepsilon ^*\lambda _i>0$. Then, $$\left( 1-\varepsilon ^*\lambda _2 \right) ^2\geqslant \left( 1-\varepsilon ^*\lambda _i \right) ^2.$$ Due to $\varepsilon ^*\lambda _2=2\left( 1-\gamma ^* \right) $, we have $$1-\varepsilon ^*\lambda _2=1-2\left( 1-\gamma ^* \right) =2\gamma ^*-1>0.$$ That is, $\gamma ^*>\frac{1}{2}$. 
%	Then, $$3\gamma ^*-1>\frac{3}{2}-1=\frac{1}{2}>0.$$
	Therefore we have 
	\begin{gather*}
		\left( \gamma ^* \right) ^2-\left( 1-\varepsilon ^*\lambda _i \right) ^2\geqslant \left( \gamma ^* \right) ^2-\left( 1-\varepsilon ^*\lambda _2 \right) ^2\\ = \left( \gamma ^*+1-\varepsilon ^*\lambda _2 \right) \left( \gamma ^*-1+\varepsilon ^*\lambda _2 \right) 
		\\
		=\left( 3\gamma ^*-1 \right) \left( 1-\gamma ^* \right) \geqslant 0.
	\end{gather*}
	It follows from $1-\gamma ^*\geqslant 0$ that
	\begin{gather*}
		\left( \gamma ^* \right) ^2-\left( 1-\varepsilon ^*\lambda _i \right) ^2+\alpha ^*\\ =\left( \gamma ^* \right) ^2-\left( 1-\varepsilon ^*\lambda _i \right) ^2+1-\gamma ^*\geqslant 0.
	\end{gather*}
	Inequality \eqref{eq: solve sub-5} is verified.
	This completes the proof of Theorem \ref{th1}. 
\end{pf}

In the following, we derive the convergence of the whole discrete dynamic system in \eqref{eq: DIG system}, which indicates the convergence of algorithm DIGing in \eqref{eq: DIGing_with_epsilon}. We also show the equilibrium point of the system in \eqref{eq: DIG system}, which implies that algorithm DIGing in \eqref{eq: DIGing_with_epsilon} converges to the optimum of problem \eqref{eq: problem average consensus}.
%Therefore, we obtain the convergence of the whole system \eqref{system-separate1}-\eqref{system-separate3} as follows.

\begin{thm}
	\label{thm: DIG with epsilon}
	Suppose Assumption \ref{ass graph} holds. Consider using algorithm DIGing in \eqref{eq: DIGing_with_epsilon} to solve problem \eqref{eq: problem average consensus}. %Let
	Denote the eigengap of the Laplacian matrix as $\kappa \triangleq \frac{\lambda _2}{\lambda _N}$, then $\kappa \in \left( 0,1 \right] $.
	Then, the worst-case optimal convergence rate of algorithm DIGing
	%	of system \eqref{eq: DIG system} 
	is $\gamma ^*=\frac{4-\kappa ^2+\kappa \sqrt{\kappa ^2+8\kappa}}{2\left( \kappa ^2+2\kappa +2 \right)}$, and the corresponding parameters are $\alpha ^*=1-\gamma ^*,\,\, \varepsilon ^*=\frac{2}{\lambda _2}\left( 1-\gamma ^* \right)$.
%	Further, we show that w
	With using the optimal parameters $\alpha^*$ and $\varepsilon^*$, the algorithm DIGing in \eqref{eq: DIGing_with_epsilon} converges to the optimum $x^*=\frac{1}{N}\sum_{i=1}^N{r_i}$ of problem \eqref{eq: problem average consensus}. 
%	Then, the discrete dynamic system of DIGing formulated in \eqref{eq: DIG system} converges to the optimum $x^*=\frac{1}{N}\sum_{i=1}^N{r_i}$. 
%	The 
\end{thm}

\begin{pf}
	In the first part of the proof, we prove the convergence of the whole system in \eqref{eq: DIG system}.
	
	From Theorem \ref{th1}, we know that as $k \rightarrow \infty$, $\tilde{x}_i\left( k \right) \rightarrow \mathbf{0}_d$, $i=2,\cdots ,N$, and $\tilde{y}_j\left( k \right) \rightarrow \mathbf{0}_d$, $j=1,\cdots ,N$, where $\mathbf{0}_d=\left[ 0,\cdots ,0 \right] _{d}^{T}$. 
	
	Consider the subsystem in \eqref{system-separate1}.
	Due to $\tilde{y}_1 \left( k \right) \rightarrow \mathbf{0}_d$, as $k \rightarrow \infty$, we have $\tilde{x}_1\left( k \right) \rightarrow \boldsymbol{c}$, as $k \rightarrow \infty$, where $\boldsymbol{c}$ is a certain constant vector.
	Therefore, the convergence rate of the subsystem \eqref{system-separate1} depends on the convergence rate of $\tilde{y}_1 \rightarrow \mathbf{0}_d$. It thus depends on the convergence rate of the subsystem in \eqref{system-separate2}-\eqref{system-separate3}.
	Thus, we obtain that the optimal convergence rate and the optimal parameter design derived in Theorem \ref{th1} hold for the whole system in \eqref{system-separate1}-\eqref{system-separate3}. Further, they also hold for the system in \eqref{eq: DIG system}.
	
%	Next, 
%	In the second part of the proof, we obtain the equilibrium point of the system in \eqref{eq: DIG system}, and thus show that algorithm DIGing in \eqref{eq: DIGing_with_epsilon} converges to the optimum of problem \eqref{eq: problem average consensus}.
	
	Next we derive the equilibrium point of the system in \eqref{system-separate1}-\eqref{system-separate3}. 
	It follows from \eqref{system-separate1} that  
	\begin{equation*}
		\begin{split}
			\tilde{x}_1\left( k+1 \right) &=\tilde{x}_1\left( k \right) -\alpha \tilde{y}_1\left( k \right) 
			\\
			&=\tilde{x}_1\left( k-1 \right) -\alpha \tilde{y}_1\left( k-1 \right) -\alpha \tilde{y}_1\left( k \right) 
			\\
			&=\cdots=\tilde{x}_1\left( 0 \right) -\alpha \sum_{n=0}^k{\tilde{y}_1\left( n \right)}.
		\end{split}
	\end{equation*}
	From \eqref{eq: DIG system unitary}, we have 
	\begin{equation*}
		\begin{split}
			\tilde{y}_1\left( k+1 \right) &=\left( 1-\alpha  \right) \tilde{y}_1\left( k \right) 
			=\left( 1-\alpha  \right) ^2\tilde{y}_1\left( k-1 \right) 
			\\
			&=\cdots =\left( 1-\alpha  \right) ^{k+1}\tilde{y}_1\left( 0 \right) .
		\end{split}
	\end{equation*}
	Then, we have
	\begin{equation*}
		\begin{split}
			\sum_{n=0}^k{\tilde{y}_1\left( n \right)}&=\sum_{n=0}^k{\left( 1-\alpha  \right) ^n}\tilde{y}_1\left( 0 \right) 
			\\
%			=\frac{1-\left( 1-\alpha L \right) ^{k+1}}{1-\left( 1-\alpha L \right)}\tilde{y}_1\left( 0 \right) 
%			\\
			&=\frac{1-\left( 1-\alpha  \right) ^{k+1}}{\alpha }\tilde{y}_1\left( 0 \right) .
		\end{split}
	\end{equation*}
%	\textcolor{blue}{If }$\alpha \in \left( 0,\frac{2}{L} \right) $, \textcolor{blue}{we have} $-1<1-\alpha L<1$. 
	Since $\alpha ^*=1-\gamma ^*$ satisfies $-1<1-\alpha ^*<1$, it is obvious that $\left( 1-\alpha^*  \right) ^k\rightarrow 0$, as $k \rightarrow \infty$.
	Then, $\sum_{n=0}^k{\tilde{y}_1\left( n \right)}\rightarrow \frac{1}{\alpha^* }\tilde{y}_1\left( 0 \right) $, as $k \rightarrow \infty$.
	
	It follows from the initialization of \eqref{eq: DIGing_with_epsilon} that $$\tilde{y}_1\left( 0 \right) =\frac{1}{\sqrt{N}}\sum_{i=1}^N{\left( x_i\left( 0 \right) -r_i \right)}=\tilde{x}_1\left( 0 \right) -\frac{1}{\sqrt{N}}\sum_{i=1}^N{r_i}.$$
	Hence $\alpha^* \sum_{n=0}^k{\tilde{y}_1\left( n \right)}\rightarrow \tilde{x}_1\left( 0 \right) -\frac{1}{\sqrt{N}}\sum_{i=1}^N{r_i}$, as $k\rightarrow \infty$.
	Then, as $k\rightarrow \infty$, $$\tilde{x}_1\left( k+1 \right) \rightarrow \frac{1}{\sqrt{N}}\sum_{i=1}^N{r_i}.$$
	
	Therefore, the equilibrium point of the subsystem in \eqref{system-separate1} is ${\tilde{x}_1}^*=\frac{1}{\sqrt{N}}\sum_{i=1}^N{r_i}$. The equilibrium point of subsystem \eqref{system-separate2}-\eqref{system-separate3} is $\tilde{x}_j^*=\mathbf{0}_d$, $j=2,\cdots ,N$, $\tilde{y}_{l}^{*}=\mathbf{0}_d$, $l=2,\cdots,N$.
	
	Further, due to $\boldsymbol{x} =U\otimes I_d  \tilde{\boldsymbol{x}} $, $\boldsymbol{y}=U\otimes I_d\tilde{\boldsymbol{y}}$, and $U=\left[ u_1,u_2,\cdots ,u_N \right] $, $u_1=\frac{1}{\sqrt{N}}\left[ 1,\cdots ,1 \right] ^T$, we have $x_{i}^{*}=\frac{1}{N}\sum_{i=1}^N{r_i}$, and $y_{i}^{*}=\mathbf{0}_d$, for $i=1,\cdots,N$. It is the equilibrium point of the system in \eqref{eq: DIG system}. It means that the algorithm DIGing in \eqref{eq: DIGing_with_epsilon} converges to the optimum $x^{*}=\frac{1}{N}\sum_{i=1}^N{r_i}$. %, for $i=1,\cdots,N$.
	
\end{pf}

\begin{Rem}
	Note that $\gamma \left( \kappa \right)=\frac{4-\kappa ^2+\kappa \sqrt{\kappa ^2+8\kappa}}{2\left( \kappa ^2+2\kappa +2 \right)}$ is monotone decreasing in $\kappa \in \left( 0,1 \right] $, where $\kappa \triangleq \frac{\lambda _2}{\lambda _N}$ is algebraic connectivity. Thus, when using the graph satisfying $\kappa =1$, \textcolor{black}{i.e. fully-connected graph,} the algorithm DIGing achieves the fastest convergence rate $\gamma _{full}^{*}=0.6$, \textcolor{black}{which is much slower than the centralized optimization. Thus, it indicates that in order to achieve better convergence performance, we have to change the iteration structure of DIGing.}
%		shows that even using the optimal designed parameters, algorithm DIGing still remains a slow convergence rate.}
\end{Rem}

It is easy to obtain that by setting $\varepsilon=1$, the optimal convergence rate of original DIGing algorithm in \eqref{eq: original DIG entry} is $$\gamma_{org} ^*=\max \left\{ 1-\frac{\lambda _2}{2},\frac{1}{2}\left( \lambda _N-1+\sqrt{2-\left( \lambda _N-1 \right) ^2} \right) \right\} ,$$ with the corresponding parameter $\alpha ^*=1-\gamma_{org} ^*$. Obviously, $\gamma_{org} ^*$ is larger than the rate $\gamma^*$ in Theorem \ref{th1}.

\section{Conclusion}
In this paper, we have presented a new approach for the parameter design of DIGing by regarding the algorithm iteration as a discrete dynamic system. We have derived the explicit formulae of the optimal worst-case convergence rate and the corresponding parameters. As the unweighted sum of squares is the simplest objective function, the results in this paper can be viewed as the first step towards optimal parameter design of DIGing for general convex objective functions. 
In future, we will further investigate how to extend the method of graph frequency decomposition and closed-loop stability analysis to analyze the convergence of general distributed optimization algorithms.

\bibliography{average_consensus_algorithms} 
\end{document}